\documentclass[reqno, 11pt]{amsart}
\usepackage{amsmath,mathtools}
 \usepackage{amssymb}
\usepackage{amsthm}
\usepackage{amsmath}
\usepackage{times}
\usepackage{latexsym}
\usepackage[mathscr]{eucal}

\numberwithin{equation}{section}
 
  \newtheorem{theorem}{Theorem}[section]
  
  \newtheorem{lemma}[theorem]{Lemma}
  \newtheorem{corollary}[theorem]{Corollary}

  \newtheorem{remark}[theorem]{Remark}
  
  \newtheorem{definition}[theorem]{Definition}
  \newtheorem{example}[theorem]{Example}

\title[A Ricci-type flow on globally null manifolds and its gradient estimates]{A Ricci-type flow on globally null manifolds and its gradient estimates}
\author[Mohamed H. A. Hamed, Fortun\'{e} Massamba and Samuel Ssekajja]{Mohamed H. A. Hamed*,  Fortun\'{e} Massamba**   and  Samuel Ssekajja***}
\newcommand{\acr}{\newline\indent}

\address{\llap{*\,} School of Mathematics, Statistics and Computer Science\acr
 University of KwaZulu-Natal\acr
 Private Bag X01, Scottsville 3209\acr 
South Africa}
\email{wdhamed82@gmail.com}  
\address{\llap{**\,} School of Mathematics, Statistics and Computer Science\acr
 University of KwaZulu-Natal\acr
 Private Bag X01, Scottsville 3209\acr 
South Africa}
\email{massfort@yahoo.fr, Massamba@ukzn.ac.za}  
\address{\llap{***\,} School of Mathematics, Statistics and Computer Science\acr
 University of KwaZulu-Natal\acr
 Private Bag X01, Scottsville 3209\acr
South Africa}
\email{ssekajja.samuel.buwaga@aims-senegal.org  } 


\begin{document}
	\begin{abstract}
		Locally, a screen integrable globally null manifold $M$ splits through a Riemannian leaf $M'$ of its screen distribution and a null curve $\mathcal{C}$ tangent to its radical distribution. The leaf $M'$ carries a lot of geometric information about $M$ and, in fact, forms a basis for the study of expanding and non-expanding horizons in black hole theory. In the present paper, we introduce a degenerate Ricci-type flow in $M'$ via the intrinsic Ricci tensor of $M$. Several new gradient estimates regarding the flow are proved. 	
		
	\end{abstract}
	\maketitle
	
	\section{Introduction}
	
	Let $M$ be a compact $m$-dimensional Riemannian manifold on which a one parameter family of Riemannian metrics $g(t)$, $t\in[0, T]$, $T< T_{\epsilon}$, where $T_{\epsilon}$ is the time where there is (possibly) a blow-up of the curvature is defined. We say $(M, g(t))$ is a solution to the Ricci flow if it is evolving by the following non-linear weakly parabolic partial differential equation \cite{AA}
	\begin{equation}
	 \partial_{t}g_{ij}(x,t)=-2Ric_{ij}(x,t),\;\;\; (x, t)\in M\times [0, T],
	\end{equation}
	with $g_{ij}(x,0)=g_{ij}(x)$, where $Ric_{ij}(x,t)$ is the Ricci curvature tensor of the evolving metric $g_{ij}(x,t)$. This evolution system was initially introduced by Hamilton in \cite{RH}. The evolution equation for the metric tensor implies the evolution equation for the curvature tensor $R$ in the form $\partial_{t}R=\Delta{R}+Q$, where $\Delta$ denotes the Laplacian operator on $M$ and $Q$ is a quadratic expression of the curvatures. In particular, the scalar curvature $\widetilde{R}$ satisfies $\partial_{t}\widetilde{R}=\Delta{\widetilde{R}}+2|\widetilde{R}ic|^{2}$, so by the maximum principle its minimum is non-decreasing along the flow. By developing a maximum principle for tensors, Hamilton \cite{RH}, proved that Ricci flow preserves the positivity of the Ricci tensor in dimension three and of the curvature operator in all dimensions; moreover, the eigenvalues of the Ricci tensor in dimension three and of the curvature operator in dimension four are getting pinched point-wisely as the curvature is getting large. In \cite{GP}, Perelman used Ricci flow and its surgery to prove Poincare Conjecture. In the papers \cite{AA} and \cite{EDLP}, the authors used Ricci flow coupled to heat-like equations to study some gradient estimates. For more information about the classical Ricci flow, see the papers \cite{AA, bena,  BC, BChin, RH,  GP,RS} and references therein.
	
	When the underlying manifold $M$ is null (sometimes called degenerate or lightlike), one may not define, in the usual way, the Ricci flow associated to the degenerate metric $g$ on $M$. In fact, it is well-known in \cite{db} that, in general, there is no Ricci tensor on $M$ via the null metric $g$. When $M$ is embedded into a semi-Riemannian manifold $\overline{M}$ as a null hypersurface, special classes of $M$ do exist with an induced symmetric Ricci tensors. For instance, in \cite{ds2} the authors shows that null hypersurfaces $\mathcal{C}[M]^{0}$ of genus zero exhibits an induced Ricci tensor. As the Ricci flow is an intrinsic geometric flow, one needs not to know much about the ambient space in which $M$ is embedded. In \cite{KDN}, Kupeli studies null manifolds using a factor bundle approach and proved the existence of many geometric objects on $M$ by taking the assumption that it is stationary, i.e., the normal bundle of $M$ is a killing distribution, which secures a Levi-Civita connection on $M$. In fact, stationary $M$ is the same as a globally null manifold studied by Duggal in \cite{duggal}. A screen integrable globally null manifold $M$ is locally a product manifold of a leaf $M'$ of its screen distribution and a null curve $\mathcal{C}$ tangent to the normal bundle of $M$. The leaves $M'$ are fundamental in studying expanding and non-expanding black hole horizons in mathematical physics, see for instance \cite{KDN} and \cite{oneil} and references therein. As the intrinsic  Ricci tensor of $M$  is symmetric on $M'$, we introduce a degenerate Ricci flow-type flow on $M$  using such a Ricci tensor and investigate its properties in terms of the associated gradient estimates.  
	
	The theory of null submanifolds of a semi-Riemannian manifold is one of the most important topics of differential geometry. More precisely, null hypersurfaces appear in general relativity as models of different types of black hole horizons \cite{db,ds2,oneil}. The study of non-degenerate submanifolds of semi-Riemannian manifolds has many similarities with the Riemannian submanifolds. However, in case the induced metric on the submanifold is degenerate, the study becomes more difficult and is strikingly different from the study of nondegenerate submanifolds \cite{db}. Some of the pioneering works on null geometry is due to Duggal-Bejancu \cite{db}, Duggal-Sahin \cite{ds2} and Kupeli \cite{KDN}. Such works motivated many other researchers to invest in the study of null submanifolds, for example, \cite{BChin, db, Jin, KDN, massa, massam, mass1, mass2, sam1} and many more references therein. The rest of the paper is organized as follows. In Section \ref{globalnull}, we review the basics on globally null manifolds. In Section \ref{ricci-type flow}, we define a degenerate Ricci-type flow evolution on a globally null manifold and give some examples. In Section  \ref{gradientest1} - \ref{gradientest2}, we develop several gradient estimates for the degenerate Ricci flow-type.

	\section{Globally null manifolds}\label{globalnull}
	
	We recall the basic concepts on globally null manifolds (see \cite{duggal} for more details and references therein). Let $(M,g)$ be a real $m$-dimensional smooth manifold where $g$ is a symmetric tensor field of type $(0,2)$. We assume that $M$ is paracompact. For $x\in M$, the \textit{radical} or \textit{null} space of $T_{x}M$ is subspace, denoted by $\mathrm{Rad}\;T_{x}M$, defined by  (see \cite{KM} for more details)
	\begin{align}\label{eq2.1}
	\mathrm{Rad}\;T_{x}M=\left\{E_{x}\in \mathrm{Rad}\;T_{x}M, g(E_{x},X)=0, \; X\in T_{x}M\right\}.
	\end{align}
	The dimension, say $r$, of $\mathrm{Rad}\;T_{x}M$ is called \textit{nullity degree} of $g$.
	$\mathrm{Rad}\;T_{x}M$ is called the radical distribution of rank $r$ on $M$. Clearly, $g$ is degenerate or non-degenerate on $M$ if and only if $r>0$ or $r=0$, respectively. We say that $(M,g)$ is a null manifold if $0<r\leq m$. 
	
	In this paper, we assume that $0<r<m$. Consider a complementary distribution $S(TM)$ to $\mathrm{Rad}\;TM$ in $TM$. We call $S(TM)$ a screen distribution on $M$,  and its existence is secured by the  paracompactness of $M$.  It is easy to see that $S(TM)$ is semi-Riemannian. Therefore, we have the following decomposition
	\begin{align}\label{eq2.2}
	TM=S(TM)\oplus\mathrm{Rad}\;TM.
	\end{align} 
	The associated quadratic form of $g$ is a mapping $h:T_{x} M\longrightarrow \mathbb{R}$ given by $h(X)=g(X,X)$, for any $X\in T_{x}M$. In general, $h$ is of type $(p,q,r)$, where $p+q+r=m$, where $q$ is the index of $g$ on $T_{x}M$. We use the following range of indices: $I$, $J \in\{1,\dots,q\}$,  $A, B \in\{q+1,\dots,q+p\}$,  $\alpha, \beta \in\{1,\dots,r\}$ and  $
	a, b \in\{r+1,\dots,m\}$, $i,j \in\{1,\dots,m\}$.
	
	Throughout the paper we consider $\Gamma(\Xi)$ to be a set of smooth sections of the vector bundle $\Xi$.
	
	Using a well-known result from linear algebra, we have the following canonical form for $h$ (with respect to a local basis of $T_{x}M$):  $h=-\sum_{I=1}^{q}(\omega^{I})^{2}+\sum_{A=q+1}^{p+q}(\omega^{A})^{2},$  where  $\omega^{1},\ldots,\omega^{p+q}$ are linearly independent local differential 1-forms on $M$. With respect to a local coordinates system $(x^{i})$, the above relation leads to 
	\begin{align*}
	 &h=-\sum_{I=1}^{q}(\omega^{I})^{2}+\sum_{A=q+1}^{q+p}(\omega^{A})^{2}, \;\; h=g_{ij}dx^{i}dx^{j}, \;\; \mathrm{rank}|g_{ij}|=p+q<m,\\
	 &g_{ij}=g(\partial_{i},\partial_{j})=-\sum_{I=1}^{q}\omega^{I}_{i}\omega^{I}_{j}+\sum_{A=q+1}^{q+p}\omega^{A}_{i}\omega^{A}_{j},
	\end{align*} 
	since $\omega^{I}=\omega^{I}_{i}dx^{i}$ and $\omega^{A}=\omega^{A}_{i}dx^{i}$.
	
	Suppose  $\mathrm{Rad}\;TM$ is an integrable distribution. Then it follows from the Frobenius theorem that leaves of  $\mathrm{Rad}\,TM$ determine a foliation on $M$ of dimension $r$, that is, $M$ is a disjoint union of connected subsets $\{L_{t}\}$ and each point $x\in M$, $M$ has a coordinate system $(\mathcal{U},x^{i})$, where $i\in\{1,\dots, m\}$ and $L_{t}\cap\mathcal{U}$ is locally given by the equation $x^{a}=c^{a}$, $a\in\{r+1,\dots, m\}$ for real constants $c^{a}$, and $(x^{\alpha})$, $\alpha\in\{1,\dots, r\}$, are local coordinates of a leaf $L$ of $\mathrm{Rad}\;TM$. Consider another coordinate system $(\overline{\mathcal{U}}, \overline{x}^{\alpha})$ on $M$. The transformation of coordinates on $M$, endowed with an integrable distribution, has the following special form.
	$0=d\overline{x}^{a}=\frac{\partial\overline{x}^{a}}{\partial\overline{x}^{b}}dx^{b}+\frac{\partial\overline{x}^{a}}{\partial\overline{x}^{\alpha}}dx^{\alpha}=\frac{\partial\overline{x}^{a}}{\partial\overline{x}^{\alpha}}dx^{\alpha}$,
	which imply $\frac{\partial\overline{x}^{a}}{\partial\overline{x}^{\alpha}}=0$,  $\forall a\in\{r+1,\dots, m\}$ and   $ \alpha\in\{1,\dots, r\}$. Hence the transformation of coordinates on $M$ is given by
	\begin{align}\label{eq2.3}
	\overline{x}^{\alpha}=\overline{x}^{\alpha}(x^{1},\dots,x^{r}),\;\;\;\; \overline{x}^{a}=\overline{x}^{a}(x^{r+1},\dots,x^{m}).
	\end{align}
	
	As $g$ is degenerate on $TM$, by using (\ref{eq2.1}) and the canonical form for $h$ we obtain $g_{\alpha\beta}=g_{\alpha a}=g_{a\alpha}=0$. Thus, the matrix of $g$ with respect to the natural frame $\{\partial_{i}\}$ becomes
	$$
	(g_{ij})=\begin{pmatrix}
	O_{r,r}& O_{r,m-r}\\
	O_{m-r,r}& g_{ab}(x^{1},\dots,x^{m})
	\end{pmatrix}.
	$$
	By the coordinates in (\ref{eq2.3}), one can show that 
	\begin{align}\label{eq2.4}
	\partial_{\alpha}g_{ab}=0,\;\;\;\; \forall a,b \in\{r+1,\dots,m\},\;\;\;\; \alpha\in\{1,\dots,r\},
	\end{align}
	holds for any other system of coordinate adapted to the foliation induced by $\mathrm{Rad}\,TM$. We, therefore, suppose that (\ref{eq2.4}) holds. Also, one can show that the screen distribution $S(TM)$ is invariant with respect to the transformations in (\ref{eq2.3}).
	
	Next, we assume that $r=1$. Thus, the 1-dimensional nullity distribution $\mathrm{Rad}\;TM$ is integrable. Using the basic formula 
	$ 
	(\mathcal{L}_{X}g)(Y,Z)= X(g(Y,Z))-g([X,Y],Z)-g(Y,[X,Z]),
	$
	for any $X,Y,Z\in\Gamma(TM)$, where  $\mathcal{L}$ is the Lie-derivative operator. The following result was established in \cite{duggal}.
	\begin{theorem}[\cite{duggal}]\label{theorem2.1}
		Let $(M,g)$ be an $m$-dimensional null manifold, with $\mathrm{Rad}\;TM$ of rank $r=1$. Then, there exists a Levi-Civita metric connection $\nabla$ on $M$ with respect to the degenerate metric tensor $g$.
	\end{theorem}	
	Let $C$ be a null curve in an $m$-dimensional null manifold $(M,g)$, with $m>1$ and locally given by $x^{i}=x^{i}(t)$, $t\in I\subset\mathbb{R}$, $i\in\{1,\dots,m\}$ for a coordinate neighborhood $\mathcal{U}$ on $C$. Then, the tangent vector field $
	\frac{d}{dt}=\left(\frac{dx^{1}}{dt},\dots,\frac{dx^{m}}{dt}\right)$,
	on $\mathcal{U}$ satisfies
	$g\left(\frac{d}{dt},\frac{d}{dt}\right)=0$,  i.e., $g_{ij}\frac{dx^{i}}{dt}\frac{dx^{j}}{dt}=0$, 
	where $g_{ij}=g(\partial_{i},\partial_{j})$ and $i,j\in\{1,\dots,m\}$. Denote by $TC$ the tangent bundle of $C$ which is a vector subbundle of $TM$, and is of rank 1. It is easy to see that $
	TC^{\bot}=\{V\in TM,\;g(V,E)=0\}=TM$, 
	where $E$ is a null vector field tangent over $C$. Suppose that $\mathrm{Rad}\;TM$ is of rank $r=1$. We consider a class of null curves such that $\mathrm{Rad}\;TM=TC$ and both are generated by a null vector field $E$ tangent over $C$. Let $S(TM)$ be the complementary screen distribution to $\mathrm{Rad}\;TM$. Then, we have 
	\begin{align}\label{eq2.10}
	TM=\mathrm{Rad}\;TM\bot S(TM)=TC\bot S(TM),
	\end{align}
	where $\bot$ means the orthogonal direct sum. It follows that $S(TM)$ is the transversal vector bundle of $C$ in $TM$. Suppose $S(TM)$ is Riemannian and $m=3$, we obtain the following differential equations
	\begin{align}\label{eq2.11}
	\nabla_{E}E=\widetilde{h}E,\;\; \nabla_{E}W_{1}=-k_{1}E+k_{3}W_{2},\;\;\nabla_{E}W_{2}=-k_{2}E-k_{3}W_{1},
	\end{align}
	where $\widetilde{h}$ and $\{k_{1}, k_{2}, k_{3}\}$ are smooth functions on $\mathcal{U}$ and $\{W_{1},W_{2}\}$ is an orthonormal basis of $\Gamma(S(TM))_{\mathcal{U}}$. We call $
	F=\left\{\frac{d}{dt}= E, W_{1}, W_{2}\right\}$  a Fr\'enet frame on $M$ along $C$ with respect to the screen distribution $S(TM)$. The functions $\{k_{1},k_{2},k_{3}\}$ and the differential equations (\ref{eq2.11}) are called curvature functions of $C$ and Fr\'enet equations for $F$, respectively. The result can be generalized for higher dimensions. It is important to mention that for any $m$, the first Fr\'enet equation for $F$ remains the same. Now we show that it is always possible to find a parameter on $C$ such that $\widetilde{h}=0$, using the same screen distribution $S(TM)$. Consider another coordinate neighborhood $\mathcal{U}^{*}$, and its Fr\'enet frame $F^{*}$, with $\mathcal{U}\cap\mathcal{U}^{*}\neq\phi$. Then, $
	\frac{d}{dt^{*}}=\frac{dt}{dt^{*}}\frac{d}{dt}$.
	
	Writing the first Fr\'enet equation in (\ref{eq2.11}) for both $F$ and $F^{*}$ and using above transformation, we obtain
	$\frac{d^{2}t}{dt^{*2}} + \widetilde{h}\left(\frac{dt}{dt^{*}}\right)^{2} = \widetilde{h}^{*}\frac{dt}{dt^{*}}$. Consider the differential equation $
	\frac{d^{2}t}{dt^{*2}} - \widetilde{h}^{*}\frac{dt}{dt^{*}}=0$,
	whose general solution comes from
	\begin{align}\label{eq2.13}
	t=a\int_{t_{0}}^{t^{*}}\exp\left(\int_{s_{0}}^{s} \widetilde{h}^{*}(t^{*})dt^{*}\right)ds+b,\;\;\;\; a,b\in\mathbb{R}.
	\end{align}
	It follows that a solution of (\ref{eq2.13}), with $a\neq 0$, may be taken as a special parameter on $C$ such that $\widetilde{h}=0$. Denote such a parameter by $p=\frac{t-b}{a}$, where we call $t$ the general parameter as given in (\ref{eq2.13}), and $p$ a distinguished  parameter of $C$. Then, the first Fr\'enet equation is given by $\nabla_{\frac{d}{dp}}\frac{d}{dp}=0$ and, therefore, $C$ is a null geodesic of $M$, with respect to the distinguished  parameter $p$. Since the first Fr\'enet equation is the same for any $m$, the following holds.
	\begin{theorem}[\cite{duggal}]
		Let $C$ be a null curve of a null manifold $M$, with $\mathrm{Rad}\;TM$ of rank 1 and Riemannian screen distribution. Then, there exists a distinguished  parameter $p$ with respect to which $C$ is a null geodesic of $M$.
	\end{theorem}
	In view of the above theorem, the author in \cite{duggal} defined globally null manifolds as follows.
	 
	A null manifold $(M,g)$ is said to be a \textit{globally null manifold} if it admits a single global null vector field and a complete Riemannian hypersurface.
	
	\noindent As a consequence, the following characterization holds on globally null manifolds.
	\begin{theorem}[\cite{duggal}]\label{global}
		Let $(M, g)$ be a globally null manifold. Then, the following assertions are equivalent:
		\begin{enumerate}
			\item  The screen distribution $S(TM)$ is integrable.
			\item  $M=M'\times C'$ is a global product manifold, where $M'$ is a leaf of $S(TM)$ and $C'$ is a 1-dimensional integral manifold of a global null curve $C$ in $M$.
			\item  $S(TM)$ is parallel with respect to the metric connection $\nabla$ on $M$.
		\end{enumerate}
	\end{theorem}
	\noindent As an example, we have the following.
	\begin{example}\label{exam}
		\rm{Let $(\mathbb{R}^{4}_{1},\overline{g})$ be the Minkowski spacetime with metric $ds^{2}=-dt^{2}+dr^{2}+r^{2}(d\theta^{2}+\sin^{2}\theta d\phi^{2})$, for a spherical coordinate system $(t,r,\theta,\phi)$, which is non-singular if we restrict $0<r<\infty,\;\; 0<\theta<\pi,\;\; 0<\phi<2\pi$. It is well known that Minkowski spacetime is globally hyperbolic. Take two null coordinates $u=t+r$ and $v=t-r\;(u>v)$. Then, we have 
			\begin{align}\label{eq2.14}
			ds^{2}=-dudv+\frac{1}{4}(u-v)^{2}(d\theta^{2}+\sin^{2}\theta d\phi^{2}),\;\;\; -\infty<u,v<\infty.
			\end{align}
			The absence of the terms $du^{2}$ and  $dv^{2}$ in (\ref{eq2.14}) implies that the two hypersurfaces $\{v=constant\}$,  $\{u=constant\}$ are null. Denote one of these null hypersurfaces by $(M,g)$, where $g$ is the induced degenerate metric tensor of $\overline{g}$. A leaf of the 2-dimensional screen distribution $S(TM)$ is topologically a 2-sphere with complete Riemannian metric $d\Omega^{2}=r^{2}(d\theta^{2}+\sin^{2}\theta d\phi^{2})$ which is the intersection of the two hypersurfaces. Since by definition a spacetime admits a global timelike vector field, it follows that both its null hypersurfaces admit a single global null vector field. Thus, there exists a pair of globally null manifolds, as null hypersurfaces of a Minkowski spacetime.				
		}
	\end{example}
	\begin{remark}
	{\rm Globally null manifolds were also studied by Kupeli in \cite{KDN}, under the name of stationary singular semi-Riemannian manifolds.}
	\end{remark}

	\section{A degenerate Ricci-type flow}\label{ricci-type flow} 
	
	Let $(M, g)$ be a null manifold and $S(T M)$ be its screen distribution. A globally null manifold $(M,g)$ is said to be a \textit{screen integrable globally null manifold} if its screen distribution $S(T M)$ is integrable.
	
	Let $(M, g)$ be a screen integrable globally null manifold and $(M',g')$ be a leaf of $S(T M)$ such that $g'=g|_{M'}$, immersed in $M$ as a non-degenerate submanifold. Denote by $P$ the projection morphism of $TM$ into $S(TM)$. Let $R$ denote the curvature tensor of $M$ with respect to the metric connection $\nabla$ on $M$. If $R'$ is the curvature tensor of $M'$ with respect to the connection $\nabla'$ on $M'$, then, in view of \cite[p. 48]{KDN}, we have $$R'(X,Y)Z'=PR(X,Y)Z,$$ for all $X,Y,Z \in \Gamma(TM $ with $Z'=PZ$. The \textit{intrinsic curvature tensor} $\overline{R}$ of $(M',g')$ is then given by 
	$ 
	\overline{R}(X',Y')Z'=R'(X,Y)Z',
	$ 
	where $X'=PX$, $Y'=PY$ and $Z'=PZ$, for all $X,Y,Z\in \Gamma(TM)$. $\overline{R}$ satisfies the usual properties of Riemannian curvature tensors (see Theorem 3.2.10 of \cite[p. 49]{KDN}). Note that $\overline{R}$ is tensorial with respect to its entries, hence for any $x\in M$, $(\overline{R}(X',Y')Z')_{x}$ depend on the values of $X',Y',Z'\in T_{x}M'$. Moreover, the associated curvature-like tensor of type (0,4) is given by 
	$$
	\overline{R}(X',Y',Z',W')=g'(\overline{R}(X',Y')Z',W'),
	$$
	where $X',Y',Z',W'\in \Gamma(T_{x}M')$.
	
	Let $(M,g)$ be a screen integrable global null submanifold. Then, the \textit{Ricci tensor} $Ric'$ of $M'$ is defined by $Ric'(X',Y')=\mathrm{trace}\{Z'\longrightarrow\overline{R}(Z',X')Y'\}$, for all $X',Y',Z'\in \Gamma(T_{x}M')$. The corresponding scalar curvature $Scal'\in C^{\infty}(M')$ of $(M',g')$ is defined by $Scal'=\mathrm{trace}(Ric')$. For a function $f\in C^{\infty}(M)$, we will denote by $\overline{f} = f_{|_{M'}}$  its restriction on $M'$. Recall that for a smooth function $\overline{f}$, we define a symmetric 2-tensor $\overline{\nabla}^{2}\overline{f}$ by $\overline{\nabla}^{2}\overline{f}(X',Y')=X'Y'\overline{f}-(\nabla_{X'}Y')\overline{f}=\mathrm{Hess}(\overline{f})(X',Y')$, from which we define the Laplacian $\underline{\Delta}\overline{f}$ of $\overline{f}$ as $\underline{\Delta}\overline{f}=\mathrm{trace}\overline{\nabla}^{2}\overline{f}=g'^{ab}(\overline{\nabla}^{2}\overline{f})_{ab}$. Notice that $\underline{\Delta}\overline{f}=\mathrm{div}(\overline{\nabla}\,\overline{f})$, where $\overline{\nabla}\,\overline{f}$ denotes the gradient of $\overline{f}$ and $\mathrm{div}(\cdot)$ denotes the divergence operator on $M'$. 
	
	For more details of curvature properties on globally null manifolds, see Kupeli \cite{KDN}. In view of the above background, we define a degenerate Ricci-type flow on $(M,g)$ as follows.
	\begin{definition}{\rm
		Let $(M, g)$ be a screen integrable globally null manifold, and let $(M',g'=g|_{M'})$ be a leaf of $S(T M)$. \textit{A degenerate Ricci-type flow} of $(M,g)$ is a family $\{g'(t)\}_{t\in I}$ of Riemannian metrics on a smooth manifold $(M',g')$, parametrized by a time interval $I\subset \mathbb{R}$ and evolving by $$\partial_{t}g'(t)=-2Ric'(g'(t)).$$}
	\end{definition}
	In harmonic local coordinates around a point $x\in M$, the Ricci tensor at $x$ is 
	$ 
	Ric'_{ab}(x)= - \frac{1}{2}\underline{\Delta}g'_{ab}(x).
	$ 
	Thus, a degenerate Ricci-type flow resembles a heat flow evolution.
	
	Next, we will give some examples of degenerate Ricci flows.
	\begin{example}
{\rm Let $(M,g)$ be a globally null manifold, and $\phi : M\longrightarrow  \mathbb{R}_{1}^{n+2}$ be an isometric immersion of $M$ into $\mathbb{R}_{1}^{n+2}$	as a null hypersurface. Then, by the definition of globally null manifolds, $M$ is embedded  in $\mathbb{R}_{1}^{n+2}$ as totally geodesic hypersurface. As $\mathbb{R}_{1}^{n+2}$ is flat, we notice that $M$ is flat as well (see \cite{db} for more details). Consequently, the intrinsic curvature $\overline{R}$ vanishes and each leaf $M'$ is flat.  Thus, the flat metric on $M'$ has zero Ricci curvature, so it does not evolve at all under the degenerate Ricci-type flow.
	}
	\end{example}
\begin{example}
		\rm{Consider the globally null manifold $(M, g)$ of Example \ref{exam}. Each leaf $M'$ of $S(TM)$ is a $2$-dimensional sphere of radius $r$, the metric is given by $g'=r^{2}g''$, where $g''=d\theta^{2}+\sin^{2}\theta d\phi^{2}$ is the metric on the unit sphere. The sectional curvatures are all $1/r^{2}$. Thus for any unit vector filed  $X'\in \Gamma(TM')$, we have $Ric'(X',X')=1/r^{2}$. Therefore, $Ric'=\frac{1}{r^{2}}g'=g''$, so the degenerate Ricci-type flow equation becomes an ordinary differential equation
			$\partial_{t}g'=-2Ric'$, from which we obtain $\frac{d(r^{2})}{d{t}}=-2$. Solving this ODE, gives $r(t)= \sqrt{R_{0}^{2}-2t}$,
			where $R_{0}$ is the initial radius of the sphere.	Notice that the degenerate Ricci-type flow of $(M,g)$ will become singular at $t=(1/2) R_{0}^{2}$. At this time, the leaf $M'$ has collapsed to a point. 
			\rm}
	\end{example}
The existence and uniqueness of degenerate Ricci-type flows can be established in the same way as in the classical Ricci flow, as well as the associated geometric evolution equations for the induced objects associated  to curvature (see \cite[p. 90]{bena} for more details). 
	
As each leaf $M'$ is Riemannian, we can define the distance function on $M'$ in a natural way. For a point $p\in M'$, define $d(x,p)$ for all $x\in M'$, where $d(\cdot, \cdot)$ is the geodesic distance. It is important to note that $d$ is only Lipschitz continuous, that is, everywhere continuous except at the cut locus of $p$ and on the point where $x$ coincides with $p$. One can easily see that $|\overline{\nabla}d|=g'^{ab}\partial_{a}d\partial_{b}d=1$ on $M'/\{\{p\}\cup cut(p)\}$. Let $d(x,y,t)$ be the geodesic distance between $x$ and $y$ with respect to the Riemannaian metric $g'(t)$, we define a smooth cut-off function $\varphi (x,t)$ with support in the geodesic cube 
$$
\mathcal{Q}_{2\rho, T}:=\{(x,y)\in M'\times (0,T]:d(x,y,t)\le 2\rho\},
$$
for any $C^{2}$-function $\psi$ on $[0,\infty)$ with $\psi(s)=1$, for $s\in  [0,1]$ and $\psi(s)=0$, for $s\in [2,+\infty]$ (see \cite{AA} for more details). Furthermore, $\psi'(s)\le 0$, $\psi''(s)\ge -c_{1}$ and $\frac{|\psi'|^{2}}{\psi}\le c_{2}$, where $c_{1},c_{2}$ are constants such that $\varphi(x,t)=\psi(d(x,p,t)/\rho)$ and $\varphi|_{\mathcal{Q}_{2\rho,T}}=1$.

As in \cite{AA}, let $M'$ be a complete $n$-dimensional leaf of the integrable secreen distribution $S(TM)$ of null manifold $(M, g)$ whose Ricci curvature is bounded from below by $Ric'\ge (n-1)k$, for some constant $k\in \mathbb{R}$. Then the Laplacian of the distance function satisfies
	\begin{align}
	\underline{\Delta}d(x,p)=\begin{cases}(n-1)\sqrt{k}\cot(\sqrt{k}\rho),& k>0,\\ (n-1)\rho^{-1},&k=0,\\(n-1)\sqrt{|k|}\coth(\sqrt{k}\rho),& k<0. \end{cases}
	\end{align}
	Throughout, we will impose boundedness condition on the Ricci curvature of the metric and note that when the metric evolves by the degenerate Ricci-type flow, boundedness and sign assumptions are preserved as long as the flow exists, so also the metrics are uniformly equivalent, if $-\rho_{1}g'\leq Ric'\leq\rho_{2}g'$, where $g'(t)$, $t\in(0,T]$ is a degenerate Ricci flow, then
	\begin{align}
	e^{-\rho_{1}T}g'(0)\leq g'(t)\leq e^{-\rho_{2}T}g'(0).
	\end{align}
	
	\section{Some gradient estimates}\label{gradientest1}
	
	In this section, we discuss the localized version of gradient estimate on the heat equation perturbed with curvature operator under both forward and backward degenerate Ricci-type flow. The estimate under backward action of degenerate Ricci-type flow is related to the local monotonicity for heat kernel and mean value theorem of Ecker, Knopf, Ni and Topping in \cite{EDLP}. They worked in general geometric flow, we follow their approach. 
	
	Let $(M, g)$ be a screen integrable globally null manifold and $(M',g'=g|_{M'})$ be a leaf of the screen distribution $S(T M)$. We consider the conjugate heat equation coupled to the backward and forward degenerate degenerate Ricci-type flow, respectively, as follows:
	\begin{align}\label{backward}
	\left\{\begin{array}{l}        
	(\partial_{t}-\underline{\Delta}+Scal')u(x,t)=0,\\
	{\partial_{t}}g'(x,t)=2Ric'(x,t)
	       \end{array}\right.
	     \;\; \mbox{and}\;\;
	\left\{\begin{array}{l}     
	  (\partial_{t}-	\underline{\Delta}+Scal')u(x,t)=0,\\
	  {\partial_{t}}g'(x,t)=-2Ric'(x,t).      
	       \end{array}\right.    
	\end{align}
	Throughout this paper, we denote by $\|\cdot\|$ the norm on $M'$ with respect to $g'$. Suppose $u=u(x,t)$ solves the conjugate heat equation and satisfies $0<u<A$ in the geodesic cube $\mathcal{Q}_{2\rho}\subset M'$ as defined by
	\begin{align}
	\mathcal{Q}_{2\rho,T}:= \{(x,t)\in M'\times(0,T]: d(x, y, t)\leq 2\rho\},
	\end{align}
	
	We set $\omega_{i}:=\partial_{i}\omega$,  $\omega_{ij}:=\partial_{i}\partial_{j}\omega$ and so on, for some smooth function $\omega$ on $M'$, where $1\le i,j\le n$. Then we have 
	\begin{theorem}\label{Theorem10}
     Let $(M, g)$ be a screen integrable globally null manifold and $(M',g')$ be a leaf of $S(T M)$. Suppose that $(M',g'(t))$ (with $t\in(0, T]$) be a complete solution to the backward degenerate Ricci flow with $Scal'\geq{-\rho_{1}}$ and $Ric'\geq{-\rho_{2}}$ and $\left \|\overline{\nabla}\,Scal'\right \|\leq{\rho_{3}}$, for some constants $\rho_{1},\rho_{2},\rho_{3}\ge 0$. Let $u=u(x,t)$ be any positive solution to the heat equation defined in $\mathcal{Q}_{2\rho,T}\subset(M'\times(0,T])$ satisfying $0<u\le A$. Then, there exist absolute constants $c_{1},c_{2}$ depending on $n$ such that 
		$$
		\frac{\left\|\overline{\nabla} u\right\|^{2}}{u^{2}}\le\left(1+\ln\Big(\frac{A}{u}\Big)\right)^{2}\left(\frac{1}{t}+c_{2}\rho_{1}+4\rho_{2}+2\rho_{3}+\frac{1}{\rho^{2}}\Big(\rho{c_{1}}\sqrt{\rho_{2}}+c_{2}\Big)\right).
		$$	
	\end{theorem}	
		\begin{proof}
Let us define $f=\ln{\frac{u}{A}}$. It is easy to see that $1-f\geq 1$. Let $\phi=\left\|\overline{\nabla}{\text{ln}}(1-f)\right\|^{2}=\frac{\left\|\overline\nabla{f}\right\|^{2}}{(1-f)^{2}}$. Then, a straightforward calculation using (\ref{backward}) gives the evolution equation of $f$  as $\partial_{t}f=\underline{\Delta}f+\left\|\overline{\nabla}{f}\right\|^{2}-Scal'$. Next, we compute the evolution equation of $\phi$. To that end,  differentiating $\phi$ with respect to $t$ and using the evolution equation of $f$, gives 
			\begin{align}
			\partial_{t}\phi=\frac{\partial}{\partial{t}}\left(\frac{\left\|\overline{\nabla}{f}\right\|^{2}}{(1-f)^{2}} \right)
			=\frac{\partial_{t}\left\|\overline{\nabla}{f}\right\|^{2}}{(1-f)^{2}}+\frac{2\left\|\overline{\nabla}{f}\right\|^{2}\partial_{t}f}{(1-f)^{3}}.\label{m1}
			\end{align}
			On the other hand, the term $\partial_{t}(\left\|\overline{\nabla}f \right\|^{2})$ is given by
			\begin{align}\label{m2}
			\partial_{t}(\left\|\overline{\nabla}f \right\|^{2})
			&=(\partial_{t}g'^{ij})\partial_{i}f\partial_{j}f+2g'(\overline{\nabla}{f},\overline{\nabla}\partial_{t}f)\nonumber\\
			&=-2Ric'_{ij}\partial_{i}f\partial_{j}f+2g'(\overline{\nabla}{f},\overline{\nabla}(\underline{\Delta}{f}+\left\|\overline{\nabla}{f}\right\|^{2}-Scal'))\nonumber\\
			&=-2Ric'_{ij}\overline{\nabla}_{i}f\overline{\nabla}_{j}f+2g'(\overline{\nabla}{f},\overline{\nabla}\,\underline{\Delta}{f})+2g'(\overline{\nabla}{f},\overline{\nabla}\left\|\overline{\nabla}{f}\right\|^{2})\nonumber\\
			&-2g'(\overline{\nabla}{f},\overline{\nabla}\,{Scal'}),
			\end{align} 	
			in which we have used (\ref{backward}) and the evolution equation of $f$. Applying (\ref{m2}) to (\ref{m1}) and using the evolution equation of $f$, gives 
			\begin{align}\label{m4}
			\partial_{t}\phi=&\frac{-2Ric'_{ij}\overline{\nabla}_{i}f\overline{\nabla}_{j}f+2g'(\overline{\nabla}{f},\overline{\nabla}\,\underline{\Delta}{f})+2g'(\overline{\nabla}{f},\overline{\nabla}\left\|\overline{\nabla}{f}\right\|^{2})-2g'(\overline{\nabla}{f},\overline{\nabla}\,Scal')}{(1-f)^{2}}\nonumber\\
			+&\frac{2\left\|\overline{\nabla}{f}\right\|^{2}(\underline{\Delta}{f}+\left\|\overline{\nabla}{f}\right\|^{2}-Scal')}{(1-f)^{3}}\nonumber\\
			=&-\frac{2Ric'_{ij}\overline{\nabla}_{i}f\overline{\nabla}_{j}f}{(1-f)^{2}}+\frac{2f_{j}f_{jji}}{(1-f)^{2}}+\frac{2g'(\overline{\nabla}{f},\overline{\nabla}\left\|\overline{\nabla}{f}\right\|^{2})}{(1-f)^{2}}-\frac{2g'(\overline{\nabla}{f},\overline{\nabla}\,{Scal'})}{(1-f)^{2}}\nonumber\\+&\frac{2\left\|\overline{\nabla}{f}\right\|^{2}\underline{\Delta}{f}}{(1-f)^{3}}+\frac{2\left\|\overline{\nabla}{f}\right\|^{4}}{(1-f)^{3}}-\frac{2Scal'\left\|\overline{\nabla}{f}\right\|^{2}}{(1-f)^{3}}.
			\end{align} 
			In view of Bochner-Weitzenb\"{o}ck identity (see \cite{EDLP} for details), we have $\underline{\Delta}(\left\|\overline{\nabla}{f}\right\|^{2})=-2f^{2}_{ij}+2g'(\overline{\nabla}{f},\overline{\nabla}\,\underline{\Delta}{f})+2Ric'_{ij}f_{i}f_{j}$, such that (\ref{m2}) reduces to 
			\begin{align}\label{m3}
			(\partial_{t}&-\underline{\Delta})\left\|\overline{\nabla}{f}\right\|^{2}\nonumber\\
			&=-2f^{2}_{ij}+2g'(\overline{\nabla}{f},\overline{\nabla}\left\|\overline{\nabla}{f}\right\|^{2})-2g'(\overline{\nabla}{f},\overline{\nabla}\,Scal')-4Ric'_{ij}\overline{\nabla}_{i}f\overline{\nabla}_{j}f.
			\end{align}
			Next, we compute $\underline{\Delta}\phi$. To that end, a direct calculation gives 
			\begin{align}\label{m7}
			\overline{\nabla}{\phi}=\frac{2g'(\overline{\nabla}{f},\overline{\nabla}\,\overline{\nabla}{f})}{(1-f)^{2}}+\frac{2\left\|\overline{\nabla}{f}\right\|^{2}\overline{\nabla}{f}}{(1-f)^{3}}=\frac{2f_{i}f_{ij}}{(1-f)^{2}}+\frac{2f_{i}^{2}f_{j}}{(1-f)^{3}}.
			\end{align}
			Applying the definition of $\underline{\Delta}$ and relation (\ref{m7}), we get
			\begin{align*}
			& \underline{\Delta}{\phi}=\left(\frac{2f_{i}f_{ij}}{(1-f)^{2}} \right)_{j}+\left(\frac{2f_{i}^{2}f_{j}}{(1-f)^{3}} \right)_{j}\\&=\frac{2f_{ij}f_{ij}+f_{i}f_{ijj}}{(1-f)^{2}}+\frac{4(1-f)f_{i}f_{ij}f_{j}}{(1-f)^{4}}+\frac{4f_{i}f_{ij}f_{j}+2f_{i}^{2}f_{jj}}{(1-f)^{3}}+\frac{6f_{i}^{2}f_{j}^{2}(1-f)}{(1-f)^{6}}\\&=\frac{2f_{ij}^{2}}{(1-f)^{2}}+\frac{2f_{i}f_{ijj}}{(1-f)^{2}}+\frac{4f_{i}f_{ij}f_{j}}{(1-f)^{3}}+\frac{4f_{i}f_{ij}f_{j}}{(1-f)^{3}}+\frac{2f_{i}^{2}f_{ij}}{(1-f)^{3}}+\frac{6f_{i}^{2}f_{j}^{2}}{(1-f)^{4}},
			\end{align*}
			from which we deduce that
			\begin{align}\label{m8}
			(\partial_{t}-\underline{\Delta})\phi&=\frac{2Ric'_{ij}f_{i}f_{j}}{(1-f)^{2}}+\frac{2f_{j}f_{jji}}{(1-f)^{2}}+\frac{2g'(\overline{\nabla}{f},\overline{\nabla}\left\|\overline{\nabla}{f}\right\|^{2})}{(1-f)^{2}}-\frac{2g'(\overline{\nabla}{f},\overline{\nabla}\,Scal')}{(1-f)^{2}}\nonumber\\
			&+\frac{2\left\|\overline{\nabla}{f}\right\|^{4}}{(1-f)^{3}}-\frac{2Scal'\left\|\overline{\nabla}{f}\right\|^{2}}{(1-f)^{3}}-\frac{2f_{ij}^{2}}{(1-f)^{2}}-\frac{2f_{i}f_{ijj}}{(1-f)^{2}}-\frac{4f_{i}f_{ij}f_{j}}{(1-f)^{3}}\nonumber\\
			&-\frac{4f_{i}f_{ij}f_{j}}{(1-f)^{2}}-\frac{6f_{i}^{2}f_{j}^{2}}{(1-f)^{4}}.
			\end{align}
			Rearranging the terms in (\ref{m8}), gives
			\begin{align}\label{m9}
			&(\partial_{t}-\underline{\Delta})\phi\nonumber\\
			&=-\frac{2}{(1-f)^{2}}\left\{f^{2}_{ij}+\frac{2f_{i}f_{ij}f_{j}}{1-f}+\frac{f^{2}_{i}f^{2}_{j}}{(1-f)^{2}}\right\}-\left\{\frac{2Ric'_{ij}f_{i}f_{j}}{(1-f)^{2}}\right. +\frac{2f_{j}f_{jji}}{(1-f)^{2}}\nonumber\\
			&\left. -\frac{2f_{i}f_{ijj}}{(1-f)^{2}}\right\}-\left\{\frac{4f_{i}f_{ij}f_{j}}{(1-f)^{3}}-\frac{4f^{2}_{i}f^{2}_{j}}{(1-f)^{4}}+\frac{2g'(\overline{\nabla}{f},\overline{\nabla}\left\|\overline{\nabla}{f}\right\|^{2})}{(1-f)^{2}}\right\}\nonumber\\
			&-\left\{\frac{2g'(\overline{\nabla}{f},\overline{\nabla}\,Scal')}{(1-f)^{2}}-\frac{2Scal'\left\|\overline{\nabla}{f}\right\|^{2}}{(1-f)^{3}}+\frac{2\left\|\overline{\nabla}{f}\right\|^{4}}{(1-f)^{3}}\right\}.
			\end{align}
			We can simplify some terms in the braces of  (\ref{m9}) further. The first term becomes
			\begin{align*}
			-\frac{2}{(1-f)^{2}}\left\{f^{2}_{ij}+\frac{2f_{i}f_{ij}f_{j}}{(1-f)}+\frac{f^{2}_{i}f^{2}_{j}}{(1-f)^{2}}\right\}=-\frac{2}{(1-f)^{2}}\Big(f_{ij}+\frac{f_{i}f_{j}}{1-f}\Big)^{2}.
			\end{align*}
			Using the Ricci identity on the next three terms we have 
			\begin{align*}
			-\frac{2Ric'_{ij}f_{i}f_{j}}{(1-f)^{2}}+\frac{2f_{j}f_{jji}}{(1-f)^{2}}-\frac{2f_{i}f_{ijj}}{(1-f)^{2}}=-\frac{4Ric'_{ij}f_{i}f_{j}}{(1-f)^{2}},
			\end{align*}
			since the Ricci identity implies that $f_{j}f_{jji}-f_{i}f_{ijj}=f_{j}(f_{jji}-f_{ijj})=-Ric'_{ij}f_{i}f_{j}$.
			Also
			\begin{align*}
			-\frac{4f_{i}f_{ij}f_{j}}{(1-f)^{3}}-\frac{4f^{2}_{i}f^{2}_{j}}{(1-f)^{4}}= \frac{-2}{1-f}f_{j}\left(\frac{2f_{i}f_{ij}}{(1-f)^{2}}+\frac{2f^{2}_{i}f_{j}}{(1-f)^{3}}\right)=\frac{-2}{1-f}g'(\overline{\nabla}{f},\overline{\nabla}{\phi}).
			\end{align*}
			Similarly, we have 
			$2(1-f)^{-2}g'(\overline{\nabla}{f},\overline{\nabla}\left\|\overline{\nabla}{f}\right\|^{2})=2g'(\overline{\nabla}{f},\overline{\nabla}{\phi})$, and
			then the second to the last three terms gives
			\begin{align*}
			&-\frac{4f_{i}f_{ij}f_{j}}{(1-f)^{3}}-\frac{4f^{2}_{i}f^{2}_{j}}{(1-f)^{4}}+\frac{2g'(\overline{\nabla}{f},\overline{\nabla}\left\|\overline{\nabla}{f}\right\|^{2})}{(1-f)^{2}}\\
			&=\frac{-2}{1-f}g'(\overline{\nabla}{f},\overline{\nabla}{\phi})+2g'(\overline{\nabla}{f},\overline{\nabla}{\phi})=\frac{-2}{1-f}g'(\overline{\nabla}{f},\overline{\nabla}{\phi}).
			\end{align*}
			Putting all these together in (\ref{m9}), we get
			\begin{align*}
			(\partial_{t}-\underline{\Delta})\phi&=-\frac{2}{(1-f)^{2}}\Big(f_{ij}+\frac{f_{i}f_{j}}{1-f}\Big)^{2}-\frac{2f}{1-f}g'(\overline{\nabla}{f},\overline{\nabla}{\phi})\\
			&-\frac{4Ric'_{ij}f_{i}f_{j}}{(1-f)^{2}}-\frac{2g'(\overline{\nabla}{f},\overline{\nabla}\,Scal')}{(1-f)^{2}}-\frac{2Scal'\left\|\overline{\nabla}{f}\right\|^{2}}{(1-f)^{3}}-\frac{2\left\|\overline{\nabla}{f}\right\|^{4}}{(1-f)^{3}}.
			\end{align*}
			Then,  using the curvature conditions we are left with the following inequality
			\begin{align}\label{m10}
			(\partial_{t}-\underline{\Delta})\phi&\le-\frac{2f}{1-f}g'(\overline{\nabla}{f},\overline{\nabla}{\phi})+\frac{2}{(1-f)}\rho_{1}\phi+4\rho_{2}\phi\nonumber\\
			&+2\rho_{3}\phi^{\frac{1}{2}}-2(1-f)\phi^{2}.
			\end{align}
			Next, we  apply a cut-off function in order to derive the desired estimate. To that end,  let $\psi$ be a smooth cut function defined on $[0,\infty)$ such that $0\le \psi(s)\le  1$, with
			$\psi^{'}(s)\le 0$, $\psi^{''}(s)\ge c_{1}$ and  $|\psi'|^{2}/\psi\le -c_{2}$, for some constants $c_{1},c_{2}>0$, depending on the dimension of the manifold only. Let us define a distance function $d(p,x)$ between the point $p$ and $x$ such that
			$\varphi(x,t)=\varphi(d(p,x,t))=\psi\left(d_{g'(t)}(p,x)/\rho\right)$, for smooth function $\varphi:M'\times(0,T]\longrightarrow \mathbb{R}$. It is easily seen that $\varphi(x,t)$ has its support in the closure of $\mathcal{Q}_{2\rho,T}$. We note that $\varphi(x,t)$ is smooth at $(y,s)\in M'\times(0,T]$, whenever point $y$ does not either coincide with $p$ or fall in the cut locus of $p$, with respect to the metric $g'(y,s)$. In what follows, we consider the function $\varphi \phi$ supported in $\mathcal{Q}_{2\rho,T}\times[0,\infty)$ is $C^{2}$ at the maxima, in which such assumption is supported  by a standard argument by Calabi known as  Calabi's trick (see \cite{EC} for more details). This approach is used in \cite{JCST}, see also \cite{EDLP, RS, QIS}. Therefore, we obtain
			\begin{align}\label{m11}
			\frac{\left \|\overline{\nabla}\varphi\right \|^{2}}{\varphi}=\frac{\left \|\psi'\right \|^{2}\left \|\overline{\nabla} d\right \|^{2}}{\rho^{2}\varphi}\leq \frac{c_{2}}{\rho^{2}},\quad\quad \frac{\partial \varphi}{\partial t}\leq c_{2}\rho_{1},
			\end{align}
			and by the Laplacian Comparison Theorem (see \cite{AA} for more details), we have 
			\begin{align}\label{f13}	\underline{\Delta} \varphi&=\frac{\psi^{'}\underline\Delta d}{\rho}+\frac{\psi^{''}\left\|\overline{\nabla}d\right\|^{2}}{\rho} \ge \frac{c_{1}}{\rho}(n-1)\sqrt{\rho_{2}}\;\coth(\sqrt{\rho_{2}}\rho)-\frac{c_{2}}{\rho^{2}},
			\end{align}
			which implies that $-\underline{\Delta} \varphi \le \frac{1}{\rho^{2}}(c_{1}\sqrt{\rho_{2}}\rho\;\coth(\sqrt{\rho_{2}}\rho)+c_{2})$.
			
			Let $(x_{0},t_{0})$ be a point in $\mathcal{Q}_{2\rho,T}$ at which $F=\varphi \phi$ attains its maximum value. At this point we have to assume that $F$ is positive, since $F=0$ implies that $\varphi \phi(x_{0},t_{0})=0$ and hence, $\phi(x,t)=0$, for all $x\in M$. Then, we $d(x,x_{0},t)<2\rho$, this yields $\overline{\nabla}u(x,t)=0$ and the theorem will follow trivially at $(x,t)$. The approach here is to estimate $(\partial_{t}-\underline{\Delta})(tF)$ and do some analyses on the result at the maximum point. To that end, we have
			\begin{align}\label{m15}
			(\partial_{t}-\underline{\Delta})(tF)&=F+t(\partial_{t}-\underline{\Delta})(\varphi \phi)\nonumber\\
			&= F-2tg'(\overline{\nabla}\varphi,\overline{\nabla} \phi)+t\varphi(\partial_{t}-\underline{\Delta})\phi+t\phi(\partial_{t}-\underline{\Delta})\varphi.
			\end{align}
			Note that at the maximum point $(x_{0},t_{0})$, we have by Derivative Test that 
			\begin{align}\label{m17}
			\underline{\Delta}F(x_{0},t_{0})=0,\;\;\; \partial_{t} F(x_{0},t_{0})\ge 0\;\; \mbox{and}\;\; \underline{\Delta}F(x_{0},t_{0})\le 0.
			\end{align}
			Taking $tF$ on $M'\times (0,T]$, we have $(\partial_{t}-\underline{\Delta})(tF)\geq 0$, whenever $(tF)$ achieves its maximum. Similarly, by this argument, we have
			$\overline{\nabla}(\varphi\phi)(x_{0},t_{0})-\phi\overline{\nabla}\varphi(x_{0},t_{0})=\varphi\overline{\nabla}\phi(x_{0},t_{0})$, which means $\varphi\overline{\nabla}\phi$ can always be replaced by $-\phi\overline{\nabla}\varphi$. By (\ref{m10}), (\ref{m15}) and (\ref{m17}), one has
			\begin{align*}
			F &-2tg'(\overline{\nabla}\varphi,\overline{\nabla}\phi)+t\varphi -2f(1-f)^{-1}g'(\overline{\nabla}f,\overline{\nabla}\phi)+2(1-f)^{-1}\rho_{1}\phi\\
			&+4\rho_{2}\phi+2\rho_{3}\phi^{\frac{1}{2}}+2(1-f)\phi^{2}+t\phi(\partial_{t}-\underline{\Delta})\varphi\ge 0.
			\end{align*}
			Taking $0\leq \varphi\leq1$ and noticing that $\frac{1}{1-f}\leq 1$, then the last inequality becomes
			\begin{align*}
			F&-2t g'\left(\frac{\overline{\nabla}\varphi}{\varphi}, \overline{\nabla}(\varphi\phi)\right)-2t\frac{\left\|\overline{\nabla}\varphi\right\|^{2}}{\varphi}\phi		+t\{2\rho_{1}F+4\rho_{2}F+2\rho_{3}F^{\frac{1}{2}}\\
			&+2(1-f)F^{2}\} -2tf(1-f)^{-1}\left\|\overline{\nabla}f\right\|\left\|\overline{\nabla}\varphi\right\|\phi+t\phi(\partial_{t}-\underline{\Delta})\varphi\ge 0.
			\end{align*}
			
			Applying the following relation $2\rho_{3}F^{\frac{1}{2}}\leq \rho_{3}F+\rho_{3}$ (see \cite{EDLP} for details) and by the Young's inequality, we derive
			\begin{align*}
			-2tf(1-f)^{-1}\left\|\overline{\nabla}f\right\|\left\|\overline{\nabla}\varphi\right\|\phi &\leq t\phi\left(\frac{\left\|\overline{\nabla}f\right\|^{2}}{1-f}\varphi +\frac{\left\|\overline{\nabla}\varphi\right\|^{2}}{\varphi}\frac{f^{2}}{1-f}\right)\\&=t(1-f)F^{2}+tF\frac{c_{2}}{\rho^{2}}\frac{f^{2}}{1-f}.
			\end{align*}
			Notice also that by bounds given in (\ref{m11}) and (\ref{f13}), we have 
			\begin{align*}
			t\phi(\partial_{t}-\underline{\Delta})\varphi &\leq t\phi\left(c_{2}\rho_{1}+\frac{1}{\rho^{2}}\Big(c_{1}\sqrt{\rho_{2}}\;\coth\;(\sqrt{\rho_{2}\rho})+c_{2}\Big)\right)\\&\leq t\phi\left( c_{2}\rho_{1}+\frac{1}{\rho^{2}}\Big(\rho c_{1}\sqrt{\rho_{2}}+c_{2}\Big)\right).
			\end{align*}
			Putting these together and dividing through by $(1-f)$, while noticing that $\frac{1}{1-f}\leq  1$ and $\frac{-f}{1-f}\leq 1$, we have 
			\begin{align*}
			F &-2t\frac{c_{2}}{\rho^{2}}F+t(2\rho_{1}+4\rho_{2}+\rho_{3})F+t\rho_{3}+tF\frac{c_{2}}{\rho^{2}}-tF^{2}\\&+t\phi\left( c_{2}\rho_{1}+\frac{1}{\rho^{2}}\Big(\rho c_{1}\sqrt{\rho_{2}}+c_{2}\Big)\right)\ge 0.
			\end{align*}
			Therefore, we have 
			\begin{align*}
			tF^{2}&\leq F+t(2\rho_{1}+4\rho_{2}+\rho_{3})F+t\left(c_{2}\rho_{1}+\frac{1}{\rho^{2}}\Big(\rho c_{1}\sqrt{\rho_{2}}+c_{2}\Big)\right)F\\
			&-t\frac{c_{2}}{\rho^{2}}F+t\rho_{3}.
			\end{align*}
			Then,
			\begin{align*}
			F^{2}&\leq F\left\lbrace \frac{1}{t}+c_{2}\rho_{1}+4\rho_{2}+\rho_{3}+\frac{1}{\rho^{2}}\Big(\rho c_{1}\sqrt{\rho_{2}}+c_{2}\Big)+\rho_{3}\right\rbrace +\rho_{3}.
			\end{align*}
			From here we can conclude that 
			\begin{align*}
			F\leq \frac{1}{t}+c_{2}\rho_{1}+4\rho_{2}+\rho_{3}+\sqrt{\rho_{3}}+\frac{1}{\rho^{2}}\Big(\rho c_{1}\sqrt{\rho_{2}}+c_{2}\Big)
			\end{align*}
			at $(x_{0},t_{0})$. Therefore,
			\begin{align}
			\left\|\overline{\nabla}f\right\|^{2}\leq(1-f)^{2}\left(\frac{1}{t}+c_{2}\rho_{1}+4\rho_{2}+2\rho_{3}+\frac{1}{\rho^{2}}\Big(\rho c_{1}\sqrt{\rho_{2}}+c_{2}\Big)\right),\nonumber
			\end{align}
			which completes the proof.
		\end{proof}
	The boundedness assumption may be weakened in the case of the forward degenerate Ricci-type flow. Therefore, we have the following.   
	\begin{theorem}\label{Theorem11}
     Let $(M, g)$ be a screen integrable globally null manifold and $(M',g')$ be a leaf of $S(T M)$. Suppose that $(M',g'(t))$ (with $t\in(0, T]$) be a complete solution to the forward degenerate Ricci-type flow with $Scal'\geq{-\rho_{1}}$ and  $\left \|\overline{\nabla}\,Scal'\right \|\leq{\rho_{3}}$, for some constants $\rho_{1}$, $\rho_{3}\ge 0$. Let $u=u(x,t)$ be any positive solution to the heat equation defined in $\mathcal{Q}_{2\rho,T}\subset(M'\times(0,T])$ satisfying $0<u\le A$. Then exist absolute constants $c_{1},c_{2}$ depending on $n$ such that
		
		$$\frac{\left\|\overline{\nabla} u\right\|^{2}}{u^{2}}\le\left(1+\ln\Big(\frac{A}{u}\Big)\right)^{2}\left(\frac{1}{t}+c_{2}\rho_{1}+2\rho_{3}+\frac{c_{2}}{\rho^{2}}\right).$$
	\end{theorem}
		\begin{proof}
		The proof of this theorem is similar to that of Theorem \ref{Theorem10}. The disparity between the estimate inequalities in Theorem \ref{Theorem10} and Theorem \ref{Theorem11} arises in some calculation which we briefly point out here. Similary, set $f=\ln\frac{u}{A}$ and $\phi=\left\|\overline{\nabla}{\text{ln}}(1-f)\right\|^{2}=\frac{\left\|\overline\nabla{f}\right\|^{2}}{(1-f)^{2}}$. Notice that $g'(t)$ evolves by the forward degenerate Ricci flow, where the inverse metric evolves as $\partial_{t}(g'^{ij})=2Ric'_{ij}$ and then 
		\begin{align}
		\partial_{t}(\left\|\overline{\nabla}f \right\|^{2})
		&=2Ric'_{ij}\overline{\nabla}_{i}f\overline{\nabla}_{j}f+2g'(\overline{\nabla}{f},\overline{\nabla}\,\underline{\Delta}{f})+2g'(\overline{\nabla}{f},\overline{\nabla}\left\|\overline{\nabla}{f}\right\|^{2})\nonumber\\&-2g'(\overline{\nabla}{f},\overline{\nabla}\,{Scal'}),  
		\end{align} 
		hence, the counterpart of (\ref{m9}) is	
		\begin{align}
		&(\partial_{t}-\underline{\Delta})\phi\nonumber\\
		&=-\frac{2}{(1-f)^{2}}\left\{f^{2}_{ij}+\frac{2f_{i}f_{ij}f_{j}}{1-f}+\frac{f^{2}_{i}f^{2}_{j}}{(1-f)^{2}}\right\}+\left\{\frac{2Ric'_{ij}f_{i}f_{j}}{(1-f)^{2}}\right. +\frac{2f_{j}f_{jji}}{(1-f)^{2}}\nonumber\\
		&\left. -\frac{2f_{i}f_{ijj}}{(1-f)^{2}}\right\}-\left\{\frac{4f_{i}f_{ij}f_{j}}{(1-f)^{3}}-\frac{4f^{2}_{i}f^{2}_{j}}{(1-f)^{4}}+\frac{2\mathcal{g}(\overline{\nabla}{f},\overline{\nabla}\left\|\overline{\nabla}{f}\right\|^{2})}{(1-f)^{2}}\right\}\nonumber\\
		&-\left\{\frac{2g'(\overline{\nabla}{f},\overline{\nabla}\,Scal')}{(1-f)^{2}}-\frac{2Scal'\left\|\overline{\nabla}{f}\right\|^{2}}{(1-f)^{3}}+\frac{2\left\|\overline{\nabla}{f}\right\|^{4}}{(1-f)^{3}}\right\}.
		\end{align}
		Now, using the Ricci identity, some braced terms  vanish. Then, we can conclude that
			\begin{align*}
			(\partial_{t}-\underline{\Delta})\phi&=-\frac{2}{(1-f)^{2}}\Big(f_{ij}+\frac{f_{i}f_{j}}{1-f}\Big)^{2}-\frac{2f}{1-f}g'(\overline{\nabla}{f},\overline{\nabla}{\phi})\\
			&+\frac{4Ric'_{ij}f_{i}f_{j}}{(1-f)^{2}}-\frac{2g'(\overline{\nabla}{f},\overline{\nabla}\,Scal')}{(1-f)^{2}}-\frac{2Scal'\left\|\overline{\nabla}{f}\right\|^{2}}{(1-f)^{3}}-\frac{2\left\|\overline{\nabla}{f}\right\|^{4}}{(1-f)^{3}}.
			\end{align*}
		Using the curvature conditions, we are left with the following inequality as in (\ref{m10})
			\begin{equation}
			(\partial_{t}-\underline{\Delta})\phi \le-\frac{2f}{1-f}g'(\overline{\nabla}{f},\overline{\nabla}{\phi})+\frac{2}{(1-f)}\rho_{1}\phi +2\rho_{3}\phi^{\frac{1}{2}}-2(1-f)\phi^{2}.\nonumber
			\end{equation}
The rest of the proof follows as in the one of Theorem \ref{Theorem10}.
		\end{proof}

	\section{Gradient estimates on forward heat equation}\label{gradientest2}

	Let $(M, g)$ be a screen integrable globally null manifold and $(M',g')$ be an $n$-dimensional complete Riemannian integral manifold of the screen distribution $S(TM)$ without boundary. Next, we discuss space-time gradient estimates for positive solutions of the forward heat equation along the degenerate Ricci flow. More precisely, we consider 
	\begin{align}\label{eq4.1}
	\partial_{t}g'(x,t)=-2Ric'(x,t),\;\;\;\;(x,t)\in M'\times(0,T], 
	\end{align}
	coupled to
	\begin{align}\label{eq4.2}
	\left(\underline{\Delta}-\partial_{t}\right)u(x,t)=0,\;\;\;\; (x,t)\in M'\times(0,T].
	\end{align}
	As mentioned in \cite{AA}, in general, our degenerate version of our estimate is local one too we also obtain it in the interior of geodesic cube. We will show how this local estimate can lead to achieving a global one. We start by proving the following lemma, which  is useful  to this section. Let us define  the geodesic cube 
	$$
	\mathcal{Q}_{2\rho,T}:= \{(x,t)\in M'\times(0,T]:d(x,x_{0},t)\leq \rho\}.
	$$
	\begin{lemma}\label{lemma4.1}
	Let $(M, g)$ be a screen integrable globally null manifold and $(M',g')$ be a leaf of $S(T M)$.  Suppose that $(M',g'(t))$ is a complete solution to the forward degenerate Ricci flow in some time interval $(0,T]$, and $-\rho_{1}g'\leq Ric'\leq\rho_{2}g'$, for some positive constants $\rho_{1}$ and $\rho_{2}$. For any smooth positive solution $u\in C^{2,1}(M'\times(0,T])$  to the heat equation in the geodesic cube $\mathcal{Q}_{2\rho,T}$, it holds that 
		\begin{align}
		(\underline{\Delta}-\partial_{t})G\geq &-2g'(\overline{\nabla}f,\overline{\nabla}G)-(\left\|\overline{\nabla}f\right\|^{2}-\alpha \partial_{t}f)-2\alpha t\rho_{1}\left\|\overline{\nabla}f\right\|^{2}\nonumber\\
		&+\frac{2\alpha}{np}t(\left\|\overline{\nabla}f\right\|^{2}-\partial_{t}f)^{2}-\frac{\alpha n q}{2}t(\rho_{1}+\rho_{2})^{2},
		\end{align}
		where $f=\ln{u}$, $G=t(\left\|\overline{\nabla}f\right\|^{2}-\alpha \partial_{t}f)$ and $\alpha\geq 1$ are given such that $\frac{1}{p}+\frac{1}{q}=\frac{1}{\alpha}$, for any real numbers $p,q> 0$.
	\end{lemma}
		\begin{proof}
		We have 
		$
		\underline{\Delta}G=t(\underline{\Delta}\left\|\overline{\nabla}f\right\|^{2}-\alpha\underline{\Delta} \partial_{t}f)$ and $\partial_{t}G=(\left\|\overline{\nabla}f\right\|^{2}-\alpha \partial_{t}f)+t\partial_{t}(\left\|\overline{\nabla}f\right\|^{2}-\alpha \partial_{t}f).
		$
		Working in local coordinates system at any point $x\in M'$ and using Einstein summation convection where repeated indices are summed up. We have by Bochner-Weitzenb\"{o}ck's identity
			$
			\underline{\Delta}\left\|\overline{\nabla}f\right\|^{2}=2f^{2}_{ij}+2f_{j}f_{jji}+2Ric'_{ij}f_{i}f_{j}.
			$
			By the hypothesis of the lemma that $g'(x,t)$ evolves by the degenerate Ricci flow we have $
			\partial_{t}(\left\|\overline{\nabla}f\right\|^{2})=(\partial_{t}g'^{ij})\partial_{i}f\partial_{j}f+2g'^{ij}\partial_{i}f\partial_{j}\partial_{t}f=2Ric'_{ij}f_{i}f_{j}+2f_{i}f_{t,i}.
			$
			Similarly, $
			\partial_{t}(\underline{\Delta}f)=(\partial_{t}g'^{ij})\partial_{i}\partial_{j}f+g'^{ij}\partial_{i}\partial_{j}\partial_{t}f\nonumber=2Ric'_{ij}f_{ij}+\underline{\Delta}(\partial_{t}f)$, which implies that $ \underline{\Delta}(\partial_{t}f)=\partial_{t}(\underline{\Delta}f)-2Ric'_{ij}f_{ij}$.
			With the above computations, we obtain the following
			\begin{align*}
			\underline{\Delta}G &=t\left(2f_{ij}^{2}+2f_{i}f_{jji}+2Ric'_{ij}f_{i}f_{j}+2\alpha Ric'_{ij}f_{ij}-\alpha\partial_{t}(\underline{\Delta}f)\right)\\&\geq t\left[\left(2f_{ij}^{2}+\alpha Ric'_{ij}f_{ij}\right)+2f_{j}f_{jji}-2\rho_{1}\left\|\overline{\nabla}f\right\|^{2}-\alpha\partial_{t}(\underline{\Delta}f)\right],
			\end{align*}
			at an arbitrary point $(x,t)\in\mathcal{Q}_{2\rho,T}$.
			
			On the other hand 
			$
			\partial_{t}G=\left(\left\|\overline{\nabla}f\right\|^{2}-\alpha\partial_{t}f\right)+t\left(2Ric'_{ij}f_{i}f_{j}+2f_{i}f_{t,i}-\alpha \partial_{t}\partial_{t}f\right).
			$
			Noticing that $f=\ln{u}$ implies the evolution
			$
			\partial_{t}f=\underline{\Delta}f+\left\|\overline{\nabla}f\right\|^{2},\; (x,t)\in\mathcal{Q}_{2\rho,T}.
			$
			Therefore, we have
			\begin{align*}
			(\underline{\Delta}-&\partial_{t})G\geq t\left[(2f_{ij}^{2}+2\alpha Ric'_{ij}f_{ij})+2f_{j}f_{jji}-2\rho_{1}\left\|\overline{\nabla}f\right\|^{2}-\alpha\partial_{t}(\underline{\Delta}f)\right]\\&-\left[(\left\|\overline{\nabla}f\right\|^{2}-\alpha\partial_{t}f)+t\partial_{t}(\left\|\overline{\nabla}f\right\|^{2}-\alpha\partial_{t}f)\right]\\&=t(2f_{ij}^{2}+2\alpha Ric'_{ij}f_{ij})-2g'(\overline{\nabla}f,\overline{\nabla}G)-(\left\|\overline{\nabla}f\right\|^{2}-\alpha \partial_{t}f)-2\alpha t\rho_{1}\left\|\overline{\nabla}f\right\|^{2}.
			\end{align*}
			Now choosing any two real numbers $p,q>0$ such that $\frac{1}{p}+\frac{1}{q}=\frac{1}{\alpha}$, we can write
			\begin{align*}
			2f_{ij}^{2}+2\alpha Ric'_{ij}f_{ij}&=\frac{2\alpha}{p}f_{ij}^{2}+2\alpha\left(\frac{1}{q}f_{ij}^{2}+ Ric'_{ij}f_{ij}\right)\geq\frac{2\alpha}{p}f_{ij}^{2}-\frac{\alpha q}{2}{Ric'}_{ij}^{2}, 
			\end{align*}
			where we have used completing the square method to arrive at the last inequality. Also by Cauchy-Schwarz inequality,
			$
			(\underline{\Delta}f)^{2}=\left\|g'^{ij}\partial_{i}\partial_{j}f\right\|^{2}\leq nf^{2}_{ij}
			$
			holds at an arbitrary point $(x,t)\in\mathcal{Q}_{2\rho,T}$, therefore we have $f^{2}_{ij}\geq\frac{1}{n}(\underline{\Delta}f)^{2}$. We can also write the boundedness condition on the degenerate Ricci curvature as $-(\rho_{1}+\rho_{2})g'\leq Ric'_{ij}\leq(\rho_{1}+\rho_{2})g'$ so that 
			$
			\underset{M}{\sup}\left\|Ric'_{ij}\right\|^{2}\leq n(\rho_{1}+\rho_{2})^{2},
			$
			since the Ricci curvature tensor is symmetric. Therefore, we have
			\begin{align}
			2t\left(f_{ij}^{2}+2\alpha Ric'_{ij}f_{ij}\right)\geq \frac{2\alpha}{np}t(\underline{\Delta}f)^{2}-\frac{\alpha n q}{2}t(\rho_{1}+\rho_{2})^{2}.
			\end{align}
			Hence the result. Our calculation is valid in the cube $\mathcal{Q}_{2\rho,T}$.
		\end{proof}
Next, we state and prove a degenerate version result for local gradient estimate (space-time) for the positive solutions to the heat equation in the geodesic cube $\mathcal{Q}_{2\rho,T}$ of bounded Ricci curvature manifold evolving by the degenerate Ricci-type flow.
	\begin{theorem}\label{theorem4.2} 
Let $(M, g)$ be a screen integrable globally null manifold and $(M',g')$ be a leaf of $S(T M)$. Suppose that $(M',g'(t))$ (with $t\in(0, T]$) is a complete solution of the degenerate Ricci-type flow (\ref{eq4.1}) such that the Ricci curvature is bounded in $\mathcal{Q}_{2\rho,T}$, i.e., $-\rho_{1}g'(x,t)\leq Ric'(x,t)\leq\rho_{2}g'(x,t)$ to some positive constants $\rho_{1}$ and $\rho_{2}$, with $(x,t)\in\mathcal{Q}_{2\rho,T}\subset(M'\times(0,T])$. If a smooth positive function $u\in C^{2,1}(M'\times(0,T])$ solves the heat equation (\ref{eq4.2}) in the geodesic cube $\mathcal{Q}_{2\rho,T}$, then for any given $\alpha>1$ with $\frac{1}{p}+\frac{1}{q}=\frac{1}{\alpha}$ and all $(x,t)\in\mathcal{Q}_{2\rho,T}\times(0,T]$, the following estimate holds
		\begin{align}\label{eq4.7}
		\underset{{x\in\mathcal{Q}_{2\rho,T}}}{\sup}\left\{ \left\|\overline{\nabla}f\right\|^{2}-\alpha \partial_{t}f\right\}&\leq \frac{\alpha n p}{4t}+c\alpha^{2}\left(\frac{\alpha^{2}p}{\rho^{2}(\alpha-1)}+\frac{1}{t}+(\rho_{1}+\rho_{2})\right)\nonumber\\&+\frac{\alpha^{2}np}{2(\alpha-1)}\rho_{1}+\frac{\alpha n}{2}(\rho_{1}+\rho_{2})\sqrt{pq}, 
		\end{align}
		where $c$ is an arbitrary constant depending only on the dimension of the leaf.
	\end{theorem}
		\begin{proof}
		Let $f=\ln u$ and $G=t(\left\|\overline{\nabla}f\right\|^{2}-\alpha\partial_{t} f)$. Then, $\frac{1}{t}G=\left(\frac{\left\|\overline{\nabla}u\right\|^{2}}{u^{2}}-\alpha\frac{u_{t}}{u}\right)$. The approach is also by using cut-off function and estimating $(\underline{\Delta}-\partial_{t})(t\varphi G)$ at the point where the maximum value for $(\varphi G)$ is attained as we did in Theorem \ref{Theorem10}. The argument follows
			\begin{align}\label{eq4.8}
			(\underline{\Delta}-\partial_{t})(t\varphi G)=2t\overline{\nabla}\varphi\overline{\nabla}G-\varphi G+t\varphi(\underline{\Delta}-\partial_{t})G+tG(\underline{\Delta}-\partial_{t})\varphi.
			\end{align}
			Suppose $(\varphi G)$ attains its maximum value at $(x_{0},t_{0})\in M'\times(0,T]$ for $t_{0}>0$. Since  $(\varphi G)(x,0)=0$ for all $x\in M$, we have by derivative test that 
			\begin{align}\label{eq4.9}
			\overline{\nabla}(\varphi G)(x_{0},t_{0})=0,\;\;\;\frac{\partial}{\partial t}(\varphi G)(x_{0},t_{0})\geq 0,\;\;\; \underline{\Delta}(\varphi G)(x_{0},t_{0})\leq 0,
			\end{align}
			where the function $(\varphi G)$ is being considered with support on $\mathcal{Q}_{2\rho,T}\times(0,T]$ and we have assumed that $(\varphi G)(x_{0},t_{0})>0$, for $t_{0}>0$. By (\ref{eq4.9}) we notice that
			$ 
			(\underline{\Delta}-\partial_{t})(\varphi G)\leq 0.
			$ 
			Using (\ref{eq4.8}), (\ref{eq4.9}) and Lemma \ref{lemma4.1}, we have 
			\begin{align*}
			0&\geq (\underline{\Delta}-\partial_{t})(t\varphi G)\geq 2t\overline{\nabla}\varphi\overline{\nabla}G-\varphi+t\varphi\Big(-2g'(\overline{\nabla}f,\overline{\nabla}G\Big)-(\left\|\overline{\nabla}f\right\|^{2}-\alpha\partial_{t}f)\nonumber\\
			&-2\alpha t\rho_{1}\left\|\overline{\nabla}f\right\|^{2}+\frac{2\alpha}{np}t\Big(\left\|\overline{\nabla}f\right\|^{2}-\partial_{t}f)^{2}-\frac{\alpha n q}{2}t\widehat{\rho}\Big)+tG(\underline{\Delta}-\partial_{t})\varphi,
			\end{align*}
			where $\widehat{\rho}=(\rho_{1}+\rho_{2})$. Noticing also that $\varphi\overline{\nabla}G$ can be replaced by $-G\overline{\nabla}\varphi$, by the condition $\overline{\nabla}(\varphi G)=0$. Therefore, we have
			\begin{align*}
			0&\leq 2t\overline{\nabla}\varphi\overline{\nabla}G-\varphi G+2t\overline{\nabla}f\overline{\nabla}\varphi.G\\&+t\varphi\left\{\frac{2\alpha}{np}t(\left\|(\overline{\nabla}f\right\|^{2}-\partial_{t}f)^{2}-2\alpha t\rho_{1}\left\|\overline{\nabla}f\right\|^{2}-\frac{\alpha n q}{2}t\widehat{\rho}\right\}\\&-t\varphi(\left\|(\overline{\nabla}f\right\|^{2}-\alpha\partial_{t}f)+tG(\underline{\Delta}-\partial_{t})\varphi,
			\end{align*}
			from which we get 
			\begin{align*}
			0\leq &-2t\frac{\left\|\overline{\nabla}\varphi\right\|^{2}}{\varphi}G-2\varphi G -2t\left\|\overline{\nabla}f\right\|^{2}\left\|\overline{\nabla}\varphi\right\|^{2}G\\&+t\varphi\left\{\frac{2\alpha}{np}t(\left\|(\overline{\nabla}f\right\|^{2}-\partial_{t}f)^{2}-2\alpha t\rho_{1}\left\|\overline{\nabla}f\right\|^{2}-\frac{\alpha n q}{2}t\widehat{\rho}\right\}\\&+tG(\underline{\Delta}-\partial_{t})\varphi.
			\end{align*}
			As we have noted earlier, Calabi's trick and Laplacian Comparison Theorem allows us to do the following calculation on the cut-off function depending on the geodesic distance, since we know that cut locus does not intersect with the geodesic cube
			$
			-\frac{\left\|\overline{\nabla}\varphi\right\|^{2}}{\varphi}\geq -\frac{c_{2}}{\rho^{2}},\;\;\; \underline{\Delta}\varphi\geq-\frac{c_{1}}{\rho}\sqrt{\rho_{1}}-\frac{c_{2}}{\rho^{2}}\;\; \text{and}\;\; -\frac{\partial\varphi}{\partial t}\geq-c_{2}\widehat{\rho}-\frac{c_{1}}{\tau},\;\;\; \tau\in(0,T].
			$
			Hence,
			$
			tG(\underline{\Delta}-\partial_{t})\varphi\geq t c_{3}\left(-\frac{1}{\rho}\sqrt{\rho_{1}}-\frac{1}{\rho^{2}}-\frac{1}{\tau}-\widehat{\rho}\right),
			$
			where we have taken $c_{3}$ to be maximum of $c_{1}, c_{2}$, so our computation becomes
			\begin{align*}
			0\geq &-2t\frac{c_{2}}{\rho^{2}}G-2\varphi G -2t\frac{\sqrt{c_{2}}}{\rho}\left\|\overline{\nabla}f\right\|\varphi^{\frac{1}{2}}G\\&+t\varphi\left\{\frac{2\alpha}{np}t(\left\|(\overline{\nabla}f\right\|^{2}-\partial_{t}f)^{2}-2\alpha t\rho_{1}\left\|\overline{\nabla}f\right\|^{2}-\frac{\alpha n q}{2}t\widehat{\rho}\right\}\\&+t c_{3}\left(-\frac{1}{\rho}\sqrt{\rho_{1}}-\frac{1}{\rho^{2}}-\frac{1}{\tau}-\widehat{\rho}\right)G.
			\end{align*}
			Multiplying through by $\varphi$ such that $0\leq\varphi\leq 1$ again, we have 
			\begin{align}\label{eq4.10}
			0\geq &-2\varphi^{2}G-2t\frac{\sqrt{c_{2}}}{\rho}\left\|\overline{\nabla}f\right\|\varphi^{\frac{3}{2}}G\nonumber\\&+\frac{2t^{2}}{n}\left\{\frac{\alpha}{p}(\varphi\left\|(\overline{\nabla}f\right\|^{2}-\varphi\partial_{t}f)^{2}-\alpha n\rho_{1}\varphi^{2}\left\|\overline{\nabla}f\right\|^{2}-\frac{\alpha n^{2} q}{4}t\widehat{\rho}\varphi^{2}\right\}\nonumber\\&+t c_{3}\left(-\frac{1}{\rho}\sqrt{\rho_{1}}-\frac{1}{\rho^{2}}-\frac{1}{\tau}-\widehat{\rho}\right)(\varphi G).
			\end{align}
			Using a standard argument from Lie and Yau \cite{PS} (see also Schoen and Yau \cite{RS} and references therein), we let   
			$
			y=\varphi\left\|\overline{\nabla}f\right\|^{2}$ and $Z=\varphi \partial_{t}f$ to have $\varphi^{2}\left\|\overline{\nabla}f\right\|^{2}=\varphi y\leq y,
			$
			and
			$
			y^{\frac{1}{2}}(y-\alpha z)=\varphi^{\frac{1}{2}}\left\|\overline{\nabla}f\right\|(\varphi\left\|\overline{\nabla}f\right\|^{2}-\varphi\partial_{t}f)=\frac{1}{t}\left\|\overline{\nabla}f\right\|\varphi^{\frac{3}{2}}G.
			$
			Noting that
			\begin{align}\label{eq4.11}
			(y-z)^{2}&=\frac{1}{\alpha^{2}}(y-\alpha z)^{2}+\frac{\alpha-1}{\alpha^{2}}y^{2}+\frac{2(\alpha-1)}{\alpha^{2}}y(y-\alpha z),
			\end{align}
			and 
			\begin{align}
			(y-z)^{2}&-np\rho_{1}y-np\frac{\sqrt{c_{2}}}{\rho}y^{\frac{1}{2}}(y-\alpha z)\nonumber\\&=\frac{1}{\alpha^{2}}(y-\alpha z)^{2}+\left(\frac{\alpha-1}{\alpha^{2}}y^{2}-np\rho_{1}y\right)\nonumber\\&+\left(\frac{2(\alpha-1)}{\alpha^{2}}y-np\frac{\sqrt{c_{2}}}{\rho}y^{\frac{1}{2}}\right)(y-\alpha z),
			\end{align}
			also using the inequality of the form $ax^{2}-bx\geq -\frac{b^{2}}{4a}$, $(a,b>0)$, we have 
			\begin{align}
			\frac{\alpha-1}{\alpha^{2}}y^{2}-np\rho_{1}y&\geq -\frac{\alpha^{2}n^{2}p^{2}\rho_{1}^{2}}{4(\alpha-1)^{2}},\\
			\mbox{and}\;\;\;\frac{2(\alpha-1)}{\alpha^{2}}y-np\frac{\sqrt{c_{2}}}{\rho}&y^{\frac{1}{2}}\geq -\frac{c_{2}\alpha^{2}n^{2}p^{2}}{8(\alpha-1)\rho^{2}}.\label{eq4.13}
			\end{align}
			Putting together (\ref{eq4.11})-(\ref{eq4.13}), the second and the third terms in the right hand side of (\ref{eq4.10}) are further calculated as follows, using $t(y-\alpha z)=\varphi G$,
			\begin{align*}
			&\frac{2t^{2}}{n}\left\{\frac{\alpha}{p}(y-z)^{2}-\alpha n\rho_{1}y-\frac{\alpha n^{2}q}{4}\widehat{\rho}\varphi^{2}-\frac{n\sqrt{c_{2}}}{\rho}y^{\frac{1}{2}}(y-\alpha z)\right\}\\ &\geq \frac{2t^{2}}{n}\left\{\left(\frac{1}{\alpha p}(y-\alpha z)^{2}-\frac{\alpha^{3}n^{2}p\rho^{2}_{1}}{4(\alpha-1)^{2}}\right)-\left(\frac{c_{2}\alpha^{3}n^{2}p}{8\rho^{2}(\alpha-1)}(y-\alpha z)+\frac{\alpha n^{2}q}{4}\widehat{\rho}\varphi^{2}\right)\right\}.
			\end{align*}
			Hence, by (\ref{eq4.10})
			\begin{align*}
			&\left\{\frac{2t^{2}}{n}\left(\frac{1}{\alpha p}(y-\alpha z)^{2}-\frac{\alpha^{3}n^{2}p\rho^{2}_{1}}{4(\alpha-1)^{2}}\right)-\frac{2t}{n}\left(\frac{c_{2}\alpha^{3}n^{2}p}{8\rho^{2}(\alpha-1)}(\varphi G)\right)-\frac{\alpha nq}{2}t^{2}\widehat{\rho}\varphi^{2}\right\}\\&+t c_{3}\left(-\frac{1}{\rho}\sqrt{\rho_{1}}-\frac{1}{\rho^{2}}-\frac{1}{\tau}-\widehat{\rho}\right)(\varphi G) -2\varphi^{2}G\le 0,
			\end{align*}
			since $t(y-\alpha z)=\varphi G$, we have 
			\begin{align*}
			0\geq \frac{2}{\alpha np}(\varphi G)^{2}&+\left[\frac{c_{4}}{\rho^{2}}t\left(-\frac{p\alpha^{3}}{(\alpha-1)}-\rho\sqrt{\rho_{1}}-1-\frac{\rho^{2}}{\tau}-\rho^{2}\widehat{\rho}\right)-1\right](\varphi G)\\&-\left(\frac{\alpha^{3}np\rho_{1}^{2}}{2(\alpha-1)^{2}}t^{2}+\frac{\alpha nq}{2}t^{2}\widehat{\rho}\varphi^{2}\right),
			\end{align*}
			where $c_{4}$ is a constant depending on $n$. We see that the left hand side of the last inequality is a quadratic polynomial in $(\varphi G)$, then using the quadratic formula and elementary inequality of the form $\sqrt{1+a^{2}+p^{2}}\leq 1+a+p$ yields
			\begin{align*}
			\varphi G &\leq \frac{\alpha np}{4}+\frac{\alpha np}{4\rho^{2}}c_{4}t\left(\frac{p\alpha^{3}}{(\alpha-1)}+\rho\sqrt{\rho_{1}}+\frac{\rho^{2}}{\tau}+\rho^{2}\widehat{\rho}\right)\\&+\frac{\alpha np}{4}\left(\frac{2\alpha\rho_{1}}{\alpha-1}t+2\widehat{\rho}^{\frac{1}{2}}\varphi t \sqrt{\frac{q}{p}}\right) =\frac{\alpha np}{4}+\frac{\alpha np}{4\rho^{2}}c_{4}t\left(\frac{p\alpha^{3}}{(\alpha-1)}+\frac{\rho^{2}}{\tau}+\rho^{2}\widehat{\rho}\right)\\
			&+\frac{\alpha^{2}np\rho_{1}}{2(\alpha-1)}t+\frac{\alpha n}{2}t\widehat{\rho}^{\frac{1}{2}}\sqrt{pq}.
			\end{align*}
			Recall that we picked up $\varphi(x,t)$ such that $0\leq \varphi\leq 1$ and particularly $\varphi(x,t)=1$ in $\mathcal{Q}_{2\rho,T}$ and sine $(x_{0},t_{0})$ is a maximum point for $(\varphi G)$ in $\mathcal{Q}_{2\rho,T}$, we have $G(x,\tau)=(\varphi G)(x,\tau)\leq(\varphi G)(x_{0},t_{0})$. Hence
			\begin{align}\nonumber
			\frac{1}{t}G(x,t)\leq\frac{\alpha np}{4t}+\frac{\alpha n}{4\rho^{2}}c_{4}\left(\frac{\alpha p}{(\alpha-1)}+\frac{\rho^{2}}{\tau}+\rho^{2}\widehat{\rho}\right)+\frac{\alpha^{2}np}{2(\alpha-1)}\rho_{1}+\frac{\alpha n}{2}\widehat{\rho}^{\frac{1}{2}}\sqrt{pq},
			\end{align}
			for all $x\in M$ such that $d(x,x_{0},\tau)<\rho$ and $\tau\in(0,T]$ was arbitrarily chosen, which implies the estimate (\ref{eq4.7}) and completes the proof.  
		\end{proof}
	Let $(M, g)$ be a screen integrable globally null manifold. Let $(M',g')$ be an $n$-dimensional compact (or noncompact without boundary) leaf of  $S(TM)$ with bounded Ricci curvature. Using previous Lemma \ref{lemma4.1} and local gradient in Theorem \ref{theorem4.2}, we now present global estimates for the positive solutions to the heat equation when the metric $g'$ evolves by the degenerate Ricci-type flow.
	\begin{theorem}
		Let $(M, g)$ be a screen integrable globally null manifold and $(M',g'(t))$ be a complete leaf of $S(T M)$ and $g'(t)$ solves the degenerate Ricci-type flow equation such that its Ricci curvature is bounded for all $(x,t)\in M'\times(0,T]$. Let $u=u(x,t)>0$ be any positive solution to the heat equation (\ref{eq4.2}). Then, we have for $-\rho_{1}g'(x,t)\leq Ric'(x,t)\leq\rho_{2}g'(x,t)$ and $\alpha>1$ with $\frac{1}{p}+\frac{1}{q}=\frac{1}{\alpha}$,  
		\begin{align}
		\frac{\left\|\overline{\nabla}u\right\|^{2}}{u^{2}}-\alpha\frac{u_{t}}{u}\leq\frac{\alpha np}{4t}+\frac{\alpha^{2} np}{2(\alpha-1)}\rho_{1}+\frac{\alpha n}{2}(\rho_{1}+\rho_{2})\sqrt{pq},
		\end{align}
		for all $(x,t)\in M'\times(0,T]$.
		Moreover, if we assume $M'$ has non-negative Ricci curvature in the case of compact manifold, i.e., for $0\leq Ric'(x,t)\leq\rho g'(x,t)$, we have 
		\begin{align}
		\frac{\left\|\overline{\nabla}u\right\|^{2}}{u^{2}}-\alpha\frac{u_{t}}{u}\leq\frac{\alpha np}{4t}+\frac{\alpha n}{2}\rho\sqrt{pq},
		\end{align} 
		for all $(x,t)\in M'\times(0,T]$ and $\alpha\geq 1$ with $\frac{1}{p}+\frac{1}{q}=\frac{1}{\alpha}$.
	\end{theorem}
    \begin{proof}
    Set $f=\ln u$ and allow $G$ to remain as before, calculating at the maximum point $(x_{0},t_{0})\in M'\times(0,T]$, we will show that a new function
    \begin{align}\label{eq4.18}
    \widetilde{G}=G-\frac{\alpha^{2}np}{2(\alpha-1)}t\rho_{1}-\frac{\alpha n}{2}t(\rho_{1}+\rho_{2})\sqrt{pq},
    \end{align}
    satisfies the inequality
    \begin{align}\label{eq4.19}
    \widetilde{G}\leq\frac{\alpha np}{4}+\frac{\alpha^{2}np}{2(\alpha-1)}t\rho_{1}+\frac{\alpha n}{2}t(\rho_{1}+\rho_{2})\sqrt{pq}.
    \end{align}
    It suffices to prove that $\widetilde{G}\leq\alpha np$, for any $(x_{0},t_{0})\in M'\times(0,T]$. We now show this by contradiction. Suppose $\widetilde{G}>\alpha np$ and $\widetilde{G}$ has its maximum at the point $(x_{0},t_{0})$, then we know that 
			$
			\overline{\nabla}\widetilde{G}(x_{0},t_{0})=0,
			$
			$
			\underline{\Delta}\widetilde{G}(x_{0},t_{0})\leq 0,
			$
			$
			\frac{\partial}{\partial t}\widetilde{G}(x_{0},t_{0})\geq 0,
			$
			and
			$			
			(\underline{\Delta}-\partial_{t})\widetilde{G}(x_{0},t_{0})\leq 0,
			$
			then, by Lemma \ref{lemma4.1}, we have $0\geq (\underline{\Delta}-\partial_{t})\widetilde{G}\geq (\underline{\Delta}-\partial_{t})G$. Noticing that 
			\begin{align*}
			(\left\|\overline{\nabla}f\right\|^{2}-\partial_{t}f)^{2}&=\frac{1}{\alpha^{2}}\left(\frac{G}{t_{0}}\right)^{2}+\frac{2(\alpha-1)}{\alpha^{2}}\left\|\overline{\nabla}f\right\|^{2}\left(\frac{G}{t_{0}}\right)+\frac{(\alpha-1)^{2}}{\alpha}\left\|\overline{\nabla}f\right\|^{4}.
			\end{align*}
			By Lemma \ref{m10}
			\begin{align*}
			0\geq (\underline{\Delta}-\partial_{t})G\geq &-2\overline{\nabla}f\overline{\nabla}G+\frac{2\alpha}{np}t_{0}(\left\|\overline{\nabla}f\right\|^{2}-\partial_{t}f)^{2}-(\left\|\overline{\nabla}f\right\|^{2}-\alpha\partial_{t}f)\\&-2\alpha t_{0}\rho_{1}\left\|\overline{\nabla}f\right\|^{2}-\frac{\alpha nq}{2}t_{0}(\rho_{1}+\rho_{2})^{2}.
			\end{align*}
			Following the calculation in Theorem \ref{theorem4.2}, we obtain
			$$
			\frac{G}{t_{0}}\leq \frac{\alpha np}{4t_{0}}+\frac{\alpha^{2} np}{2(\alpha-1)}\rho_{1}+\frac{\alpha n}{2}(\rho_{1}+\rho_{2})^{2}\sqrt{pq},
			$$
			after sending $\rho\;\; \text{to}\;\;\infty$. Consequently, we have the following inequality
			$$
			0\geq\frac{2t_{0}}{\alpha np}\left(\frac{G}{t_{0}}\right)^{2}-\frac{G}{t_{0}}-\left(\frac{\alpha^{2}np}{2(\alpha-1)}\rho_{1}+\frac{\alpha n}{2}(\rho_{1}+\rho_{2})^{2}\sqrt{pq}\right),
			$$
			resulting into a quadratic inequality, and since from (\ref{eq4.18})
			$$
			\frac{G}{t_{0}}=\frac{\widetilde{G}}{t_{0}}+\frac{\alpha^{2}np}{2(\alpha-1)}\rho_{1}+\frac{\alpha n}{2}(\rho_{1}+\rho_{2})^{2}\sqrt{pq},
			$$
			\begin{align}
			\frac{2t_{0}}{\alpha np}\left(\frac{G}{t_{0}}\right)^{2}-\frac{G}{t_{0}}-\frac{\alpha^{2}np}{2(\alpha-1)}\rho_{1}\leq 0.
			\end{align}
			Using the quadratic formula, we have 
			$$
			\frac{G}{t_{0}}\leq \frac{\alpha np}{4t_{0}}\left\{ 1+\sqrt{1+\frac{4\alpha np}{(\alpha-1)}}\rho_{1}\right\},
			$$
			which obviously implies that $\widetilde{G}\leq \alpha np$, a contradiction by the assumption (\ref{eq4.19}). By definition of $\widetilde{G}$, we therefore have 
			$$
			\frac{G}{t_{0}}=\frac{\left\|\overline{\nabla}u\right\|^{2}}{u^{2}}-\frac{u_{t}}{u}\leq\frac{\widetilde{G}}{t_{0}}+\frac{\alpha^{2} np}{2(\alpha-1)}\rho_{1}+\frac{\alpha n}{2}(\rho_{1}+\rho_{2})\sqrt{pq}.
			$$
			The desired estimate follows since $t_{0}$ was arbitrarily chosen.
		\end{proof}
		The case $\alpha=1$ leads to the following result, in which we choose $p=q=2$.
\begin{theorem}
	Let $(M, g)$ be a compact screen integrable globally null manifold and $(M',g'(t))$ (with $t\in(0, T]$) be a complete leaf of $S(T M)$ and $g'(t)$ solves the degenerate Ricci-type flow equation (\ref{eq4.2}). Assume that $0\leq Ric'(x,t)\leq\rho g'(x,t)$ for positive constant $\rho>0$	and for all $(x,t)\in M'\times (0,T]$. Let $u=u(x,t)>0$ be any positive solution to the heat equation (\ref{eq4.2}). Then, the estimate 
	\begin{align}
	\frac{\left\|\overline{\nabla}u\right\|^{2}}{u^{2}}-\frac{u_{t}}{u}\leq\frac{n}{2t}+n\rho,
	\end{align}
	holds, for all $(x,t)\in M'\times(0,T]$.
	\begin{proof}
		As before, set $f=\ln u$ and $G=t\left(\left\|\overline{\nabla}f\right\|^{2}-\partial_{t} f\right)$.
		Fix $\tau\in(0,T]$ and we can choose a point $(x_{0}, t_{0})\in M'\times(0,\tau]$, where $G_{1}$ attains its maximum on $M'\times(0,\tau]$. Now we can show that 
		\begin{align}\label{equ5.20}
		G(x_{0},t_{0})\leq t_{0}n\rho+\frac{n}{2}.
		\end{align}
		If $t_{0}=0$, then $G(x_{0},t_{0})=0$, for every $x\in M$. From this, we can get estimate (\ref{equ5.20}).
		Next, we can assume $t_{0}>0$, from Lemma \ref{lemma4.1} and the conditions on the Ricci curvature of $M$ imply the inequality
		\begin{align*}
		(\underline{\Delta}-\partial_{t})G\geq -2\overline{g}(\overline{\nabla}f,\overline{\nabla}G)+\frac{2p}{n}\frac{G^{2}}{t_{0}}-\frac{G}{t_{0}}-\frac{t_{0}n}{2(1-p)}\rho^{2}.
		\end{align*}
		Recall that $G$ attains its maximum at $(x_{0},t_{0})$ \cite{MXA}. This gives us 
		 $$
		\underline{\Delta}G(x_{0},t_{0})\leq 0,\;\;\;
		\frac{\partial}{\partial t}G(x_{0},t_{0})\geq 0,\;\;\;
		\overline{\nabla}G(x_{0},t_{0})=0. 
		$$
		Then, the estimate 
		$$
		\frac{2p}{n}\frac{G^{2}}{t_{0}}-\frac{G}{t_{0}}-\frac{t_{0}n}{2(1-p)}\rho^{2}\leq 0,
		$$
		holds at $(x_{0},t_{0})$, and the quadratic formula yields
		$$
		G(x_{0},t_{0})\leq  \frac{n}{4p}\left\{ 1+\sqrt{1+\frac{4 pt_{0}^{2}}{(1-p)}}\rho^{2}\right\}.
		$$
		Now recall that $(x_{0},t_{0})$ is a maximum point for $G$ on $M'\times(0,\tau]$ from this fact, we can get that
		$$
		G(x,\tau)\leq G(x_{0},t_{0})\leq t_{0}\rho n +\frac{n}{2}\leq \tau\rho n+\frac{n}{2},
		$$
		for all $x\in M$. Therefore, the estimate
		$$
		\frac{\left\|\overline{\nabla}u\right\|^{2}}{u^{2}}-\frac{u_{t}}{u}\leq\rho n+\frac{n}{2\tau},
		$$
		holds at $(x,\tau)$. Since the number $\tau\in(0,T]$ can be chosen arbitrarily.
	\end{proof}
\end{theorem}
\begin{corollary}	
Let $(M, g)$ be a screen integrable globally null manifold and $(M',g')$ be a complete leaf of $S(T M)$. Assume that  $g'(t)$ solves the degenerate Ricci-type flow equation (\ref{eq4.1}) such that its Ricci curvature is bounded for all $(x,t)\in M'\times(0,T]$. Let $u=u(x,t)>0$ be any positive solution to the heat equation (\ref{eq4.2}), and  $0<u\le A$. Then, for $-\rho_{1}g'(x,t)\leq Ric'(x,t)\leq\rho_{2}g'(x,t)$, there exists an absolute constant $C$ such that 
$$
t\left\|\overline{\nabla}u\right\|^{2}\le CA(1+\rho_{1}T),
$$
in  $M'\times(0,T]$.
\end{corollary}
Also, we have the following.
\begin{corollary}
Let $(M, g)$ be a screen integrable globally null manifold and $(M',g')$ be a complete leaf of $S(T M)$. Assume that $(M',g'(t))$ be a complete solution of the degenerate Ricci-type flow (\ref{eq4.1}) such that the Ricci curvature satisfies $$-\rho_{1}g'(x,t)\leq Ric'(x,t)\leq\rho_{2}g'(x,t),$$ for all $(x,t)\in M'\times(0,T]$. Let $u=u(x,t)$ be any positive solution to the heat equation (\ref{eq4.2}). Then, for $\alpha>1$ with $\frac{1}{p}+\frac{1}{q}=\frac{1}{\alpha}$, we have
\begin{align}
\frac{\left\|\overline{\nabla}u\right\|^{2}}{u^{2}}-\frac{u_{t}}{u}\leq \frac{\alpha np}{4t}+c(n)\alpha^{2}(\rho_{1}+\rho_{2}),
\end{align}
for all $(x,t)\in M'\times(0,T]$ and absolute constant $c(n)$ depending only on $n$.
\end{corollary}

	\section*{Acknowledgments}
	
	Mohamed H. A. Hamed and Samuel Ssekajja would like to thank the Simons Foundation through the RGSM-Project for financial support. 
	This work is based on the research supported wholly / in part by the National Research Foundation of South Africa (Grant Numbers: 95931 and 106072).

\end{document}